\documentclass[11pt]{amsart}

\usepackage{amsmath, amsfonts, amssymb, amscd, amstext, amsthm}
\usepackage{latexsym}
\usepackage{mathrsfs}
\usepackage{mathscinet}

\usepackage{stmaryrd}
\usepackage[utf8]{inputenc}
\usepackage[all]{xy}
\SelectTips{cm}{}

\usepackage[dvipsnames]{xcolor}
\usepackage{tikz}

\usepackage{url}
\usepackage[margin=3cm]{geometry}

\usepackage{marginnote}
\long\def\@savemarbox#1#2{\global\setbox#1\vtop{\hsize\marginparwidth 
  \@parboxrestore\tiny\raggedright #2}}
 \usepackage[pagebackref,%colorlinks,linkcolor=black,citecolor=black%
    ,pdfborder={0 0 0}
    ,urlcolor=black,hypertexnames=false]{hyperref}
\hypersetup{pdfauthor={Lizhi Chen},pdftitle={Topological complexity and finite domination}}

\AtBeginDocument{%
   \def\MR#1{}
}

\makeatletter
\@namedef{subjclassname@2020}{\textup{2020} Mathematics Subject Classification}
\makeatother

\newtheorem{theorem}{Theorem}[section]
\newtheorem{lemma}[theorem]{Lemma}
\newtheorem{proposition}[theorem]{Proposition}

\theoremstyle{definition}
\newtheorem{definition}[theorem]{Definition}

\theoremstyle{remark}
\newtheorem{remark}[theorem]{Remark}

\numberwithin{equation}{section}

\newcommand{\ts}{\hspace{.11111em}}
\newcommand{\tts}{\hspace{.05555em}}
\newcommand{\Ncal}{\mathcal{N}}
\newcommand{\Ucal}{\mathcal{U}}
\newcommand{\Vcal}{\mathcal{V}}

\newcommand{\R}{\mathbb{R}}

\newcommand{\Z}{\mathbb{Z}}

\DeclareMathOperator{\dist}{\operatorname{\mathsf{dist}}\tts}

\DeclareMathOperator{\vol}{\operatorname{\mathsf{Vol}}\tts}
\DeclareMathOperator{\inj}{\operatorname{\mathsf{inj}}\tts} % Injectivity radius
\DeclareMathOperator{\emb}{\operatorname{\mathsf{Emb}}\tts} % Embolic volume 
\newcommand{\RP}{\mathbb{R}\mathbb{P}}

\DeclareMathOperator{\sys}{\operatorname{\mathsf{Sys}}\tts}

\DeclareMathOperator{\FillRad}{\operatorname{\mathsf{Fill} \ts \mathsf{ Rad}}\tts}

\DeclareMathOperator{\dens}{\operatorname{\mathsf{Density}}\tts}

\begin{document}
%\doublespacing 
% \linenumbers
 
\title{Topological Complexity and Finite Domination}
\author[L.~Chen]{Lizhi Chen}
\address{
\hspace*{0.055in}School of Mathematics and Statistics, Lanzhou University \newline
\hspace*{0.175in} Lanzhou 730000, P.R. China 
}

\email{\hspace*{0.025in} lizhi.chen.math@gmail.com}

\thanks{Supported by Youth Scientists Fund of NSFC (grant No. 11901261) and the Fundamental Research Funds for the Central Universities (grant No. lzujbky-2017-26) and NSFC (grant No. 12271225).}



\date{\today}

\begin{abstract} 
Let $M$ be a closed, connected, smooth $n$-dimensional manifold. We prove that $M$ is dominated by the underlying space of the $n$-skeleton of a finite simplicial complex. Furthermore, the total number of simplices in the $n$-skeleton is bounded above by a constant depending only on $n$ and the embolic volume of $M$. 
\end{abstract}

\maketitle

\section{Introduction}
Gromov's Betti number theorem gives an upper bound for the total Betti number of a closed Riemannian manifold in terms of a lower bound on its sectional curvature and an upper bound on its diameter. Weiss~\cite{Weiss1996} showed a stronger version of this theorem in terms of domination. In the study of systolic inequalities, Gromov also established a curvature-free upper bound for the total Betti number. In this note, we show that Gromov's curvature-free upper bound admits a generalization in terms of domination.

For compact Riemannian manifolds of sectional curvature bounded from below, Gromov's Betti number theorem states that the total Betti number is bounded above by a positive constant depending on the dimension, a lower bound on sectional curvature and the diameter, see \cite{Gromov1981} or \cite[Theorem A]{Weiss1996}. In \cite{Weiss1996}, Weiss modified Gromov's argument to show that the total Betti number in Gromov's theorem can be replaced by the number of cells in a finite dominating CW-complex.  
\begin{theorem}[\cite{Weiss1996}]
 Let $M$ be a closed, connected, n-dimensional Riemannian manifold, with sectional curvature bounded from below by $-\kappa^2$ ($\kappa > 0$). Then there exists a constant $\mathcal{C} = \mathcal{C}(n)$, depending only on $n$, such that the manifold $M$ is dominated by a CW-complex with number of cells at most
 \begin{equation*}
  \mathcal{C}^{1 + \kappa D} ,
 \end{equation*}
 where $D$ is the diameter of $M$. 
\end{theorem}

Two notions of \emph{domination} are commonly used. Following \cite[Definition 2.1]{Harpe2017}, an oriented connected closed $n$-dimensional manifold $M$ is dominated by another oriented connected closed $n$-dimensional manifold $N$ if there exists a continuous map $f: N \to M$ of nonzero degree. Domination in the sense of 
Whitehead is a different notion; see Whitehead \cite{Whitehead1948} or Hatcher \cite[Appendix]{Hatcher2002}. Throughout this paper, unless otherwise stated, we use the notion of domination in Whitehead's sense. 
\begin{definition}[Whitehead domination]
 A topological space $X$ is said to be dominated by a CW-complex $Y$ if there exist continuous maps $r: X \to Y$ and $s: Y \to X$, such that $s \circ r: X \to X$ is homotopic to the identity map. 
\end{definition}

Let $M$ be an $n$-dimensional manifold. 
For a Riemannian metric $g$ defined on $M$, denote by $\vol_{g}(M)$ the Riemannian volume, and by $\inj(M, g)$ the injectivity radius. The \emph{embolic volume} of $M$, denoted $\emb(M)$, is defined by
\begin{equation*}
 \inf_{g} \frac{\vol_{g}(M)}{\inj(M, g)^n},
\end{equation*}
where the infimum is taken over all Riemannian metrics $g$ on $M$. Berger's embolic inequality (see (\ref{Berger_emb}) in the following) implies that $\emb(M)$ is positive if $M$ is compact. The main result of this paper is as follows.

\begin{theorem} \label{thm_main}
 Let $M$ be a closed, connected, and smooth $n$-dimensional manifold. There exist positive constants $C_n$ and $C_n^{\prime}$, depending only on $n$, such that $M$ is dominated by the underlying space of the $n$-skeleton $\Ncal^{(n)}$  
 of a finite simplicial complex $\Ncal$. Furthermore, the total number of simplices in the $n$-skeleton $\Ncal^{(n)}$ is at most
 \begin{equation} \label{bound_emb}
  C_n ( 1 + \emb(M) )^{n+1} \exp{ \left( C_n^{\prime} \sqrt{\log{(1+\emb(M))} } \right) }. 
 \end{equation}
\end{theorem}

In \cite[Section 6.4.C.]{Gromov1983}, Gromov proved that the total Betti number of a closed aspherical manifold is bounded above by a positive constant depending only on its systolic volume and dimension. The same result holds for embolic volume, since embolic volume is an upper bound for systolic volume. 
Theorem~\ref{thm_main} strengthens this conclusion by controlling the size of a finite simplicial complex dominating the manifold, in analogy with Weiss's theorem in \cite{Weiss1996}.

The quantitative bound (\ref{bound_emb}) in Theorem~\ref{thm_main} is obtained through Gromov's covering techniques. The same approach can be applied to establish the following bound on the homological complexity of closed essential manifolds. A closed $n$-dimensional manifold $M$ is \emph{essential} if there exists a continuous map $f: M \to K$ from $M$ to an aspherical topological space $K$, such that the image $f_{*}([M]) $ of the fundamental class $[M]$ is nontrivial in $H_{n}(K; G)$, with $G = \Z$ if $M$ is orientable, and $G = \Z_2$ if $M$ is non-orientable.  
 
\begin{theorem}
 \label{thm_02}
 Let $M$ be a closed, essential, smooth $n$-dimensional manifold. 
Then there exist a finite simplicial complex $\mathcal{N}$ and a continuous map $f: M \to |\mathcal{N}^{(n)}|$ from $M$ to the underlying space of the $n$-skeleton $\Ncal^{(n)}$, such that $f_{*}([M]) \neq 0$ in $H_n (|\mathcal{N}^{(n)}|; G)$, with $G=\Z$ if $M$ is orientable, and $G=\Z_2$ if $M$ is non-orientable. 
Moreover, the total number of simplices of $\mathcal{N}^{(n)}$ is at most 
  \begin{equation} \label{est_compact}
   C_n (1+\emb(M))^{n+1} \exp{ \left( C_n^{\prime} \sqrt{\log{(1+\emb(M))}} \right) } , 
  \end{equation}
 where $C_n$ and $C_n^{\prime}$ are positive constants depending only on $n$.  
\end{theorem}

Let $(M, g)$ be an $n$-dimensional Riemannian manifold. 
Berger's embolic inequality states that
\begin{equation} \label{Berger_emb}
 \inj(M, g)^n \leq c(n) \vol_{g}(M),
\end{equation}
where $c(n)$ is a constant depending only on $n$ (see \cite[Section 7.2.4.]{Berger2003} and the references therein). For a compact manifold $M$, the reciprocal $1/\emb(M)$ is the optimal constant in the corresponding embolic inequality over all Riemannian metrics on $M$. 
Determining exact embolic volumes is generally an open problem. Based on Croke's work, we know that the smallest $c(n)$ is equal to the reciprocal of the embolic volume of the $n$-sphere $S^n$,
\begin{equation*}
 \emb(M) \geq \emb(S^n) .
\end{equation*}
Moreover, $\emb(S^n)$ is known to be equal to $\frac{\sigma_n}{\pi^n}$, where $\sigma_n$ is the volume of the standard Euclidean $n$-sphere. According to \cite[Section 7.2.4.1, p. 357]{Berger2003}, the $n$-sphere $S^n$ is the only compact manifold for which the precise embolic volume is known. 
In \cite[Proposition 14]{Croke1980} (also see Berger~\cite[Theorem 149, p. 355]{Berger2003}), Croke proved a local embolic inequality: if $0 < R \leq \frac{1}{2} \inj(M, g) $, then 
\begin{equation}\label{Croke}
 \vol_{g} ( B(p, R) ) \geq \beta_n R^n
\end{equation}
holds for every ball $B(p, R)$ centered at $p$ with radius $R$ in a compact $n$-dimensional Riemannian manifold $(M, g)$, where $\beta_n$ is a constant depending only on $n$. A non-optimal value for $\beta_n$ is given by Croke in \cite{Croke1980} (also see \cite[Section 7.2.4.]{Berger2003}).

 \bigskip

\noindent \textbf{Acknowledgements.} The author is grateful to Hengyu Zhou for helpful discussions and to the anonymous reviewers for their valuable comments and suggestions.

\section{Nerve and filling radius}
\label{sec:nerve_FillRad}

Let $X$ be a topological space, and let $\Ucal$ be an open cover of $X$.
\begin{definition}
 The nerve of $\Ucal$, denoted by $\Ncal (\Ucal)$, is the simplicial complex defined as follows:
 \begin{enumerate}
  \item For each $U \in \Ucal$, there is a corresponding vertex $u$; 
  \item The vertices $u_{i_0}, u_{i_1}, \cdots , u_{i_k}$ span a $k$-simplex if and only if $U_{i_0} \cap U_{i_1} \cap \cdots \cap U_{i_k} \neq \emptyset$.
 \end{enumerate}
\end{definition}

We denote by $| \Ncal (\Ucal) |$ the underlying polyhedron of the nerve $\Ncal (\Ucal)$. 
Suppose that $\{ \varphi_U : X \to [0, 1] \, | \, U \in \Ucal \}$
is a locally finite partition of unity subordinate to $\Ucal$. The associated \emph{Alexandrov map} $ \Phi: X \to | \Ncal (\Ucal) |  $ is defined by 
\[ \Phi(x) = \sum_{U \in \Ucal} \varphi_{U} (x) \ts \ts u . \]
Since the partition of unity is locally finite, the sum is finite at each point. 
Moreover, $\Phi: X \to  | \Ncal (\Ucal) | $ is continuous. We refer to \cite[VIII, \textsection 5]{Dugundji1966}  for more details regarding the Alexandrov map.

Let $(M, g)$ be a closed Riemannian manifold, and let $\mathrm{dist}_g$ denote the distance function induced by $g$. 
Let $L^{\infty}(M)$ be the space of all bounded Borel functions on $M$. On $L^{\infty}(M)$, we define the supremum norm,
 \begin{equation*}
  \| f \|_{L^{\infty}} = \sup_{x\in M} | f(x) | . 
 \end{equation*}
The metric space $(M, \mathrm{dist}_{g})$ is embedded into $L^{\infty}(M)$ by the Kuratowski embedding $K: M \to L^{\infty}(M)$, which is defined by
\begin{equation*}
 K(x)(y) = \mathrm{dist}_{g}(x, y) 
\end{equation*}
for all $x, y \in M$.

Let $G$ be the coefficient group for the homology. The fundamental class  $[M]$ is the generator of $H_n(M; G)$, with $G = \Z$ if $M$ is orientable, and $G = \Z_2$ if $M$ is non-orientable. Denote by $N_r(K(M))$ the $r$-neighborhood of $K(M)$ in $L^{\infty}(M)$.

\begin{definition}
The \emph{filling radius} of $(M, g)$, denoted $\FillRad(M, g)$, is defined to be the infimum of $r > 0$ such that the image $K_{*}([M])$ of the fundamental class vanishes in $ H_n( N_r ( K(M) ) ; G)$.
\end{definition}

Gromov \cite{Gromov1983} proved systolic inequalities on closed essential manifolds using filling radius. 
\begin{definition}
 A closed $n$-dimensional manifold $M$ is \emph{essential} if there exists a continuous map $f: M \to X$ from $M$ to an aspherical topological space $X$, such that $f_{*}([M]) \neq 0$ in $H_n(X; G)$. Here the coefficient group $G=\Z$ if $M$ is orientable, and $G=\Z_2$ if $M$ is non-orientable.
\end{definition}

\begin{remark}
 All aspherical closed $n$-manifolds are essential. Moreover, $n$-dimensional projective space $\RP^n$ is also essential. 
\end{remark}

In a Riemannian manifold $(M, g)$, the homotopy 1-systole $\sys \pi_1(M, g)$ is defined to be the length of a shortest non-contractible closed curve. By definition, the homotopy 1-systole is related to the injectivity radius by $2 \inj(M, g) \leq \sys \pi_1(M, g)$.    

\begin{theorem}[{Gromov~\cite[Lemma 1.2.B.]{Gromov1983}}, also see {\cite[Theorem 3.3.]{crokeK2002_universal}}] \label{thm:FillRad_Sys}
Let $(M, g)$ be a closed essential Riemannian manifold. Then 
  \begin{equation*}
   \sys \pi_1(M, g) \leq 6 \, \FillRad(M, g).
  \end{equation*}
\end{theorem}

The following lemma is a simplicial version of \cite[Lemma 6]{Guth2011}. 

\begin{lemma} \label{Guth}
 Let $\Vcal = \{ B(p_j, R_j) \}_{j=1}^{N}$ be an open cover of
 a closed Riemannian manifold $M$ by metric balls,  
  where $R_j \leq R_0$ and $R_0 > 0$. 
If the Alexandrov map $\Phi: M \to | \Ncal (\Vcal) | $ satisfies ${\Phi}_{*}([M]) = 0$, then the filling radius of $M$ must be at most $R_0$.  
\end{lemma}

In \cite{Guth2011}, Guth uses the rectangular nerve rather than the simplicial nerve. The proof of Lemma \ref{Guth} is obtained by adapting Guth's argument in our setting. 

\begin{proof}
Using the Kuratowski embedding $K: M \to L^{\infty}(M)$, we define an affine map
 \[ \Psi \colon |\Ncal(\Vcal)| \longrightarrow L^{\infty}(M) \]
by sending each vertex $u_j \in \Ncal(\Vcal)$ to $K(p_j)$, and extending $\Psi$ simplicially on $|\Ncal(\Vcal)|$. If $x = \sum_{i=1}^m t_i u_{j_i}$ lies in a simplex, choose
 \[ q \in \bigcap_{i=1}^m B(p_{j_i}, R_{j_i}). \]
Then
 \begin{align*}
   \| \Psi(x) - K(q) \|_{\infty} = & \left\| \sum_{i=1}^m t_i K(p_{j_i}) - K(q) \right\|_{\infty}  \\
     \leq & \sum_{i=1}^m t_i \left\| K(p_{j_i}) - K(q) \right\|_{\infty} \\  
     = & \sum_{i=1}^m  t_i \dist_g(p_{j_i}, q) \\
    < & \sum_{i=1}^m  t_i R_{j_i} \\ 
    \leq & R_0. 
 \end{align*}
Thus $\Psi(|\Ncal(\Vcal)|) \subseteq N_{R_0}(K(M))$. Similarly, for every $p \in M$, $ \| \Psi(\Phi(p)) - K(p) \|_{\infty} < R_0$, 
so $\Psi \circ \Phi$ and $K$ are joined by a straight-line homotopy inside the neighborhood $N_{R_0}(K(M))$. If $\Phi_{*}([M]) = 0$, then $K_{*}([M]) = 0$ in $H_n(N_{R_0}(K(M)); G)$. This proves the lemma.  
\end{proof}

The following proposition is used in the proof of Theorem~\ref{thm_02}.

\begin{proposition}\label{nerve_homology}
 Let $(M, g)$ be a closed essential $n$-dimensional Riemannian manifold. 
For a cover $\Vcal = \{ B(p_j, R_j) \}_{j=1}^N$ of open metric balls, let the radii $R_j$ be at most $\frac{1}{8} \inj (M, g)$. Then the Alexandrov map $\Phi: M \to | \Ncal (\Vcal) |$ satisfies $\Phi_{*} ( [M] ) \neq 0 $. 
\end{proposition} 

\begin{proof}
 Suppose that $\Phi_{*}([M]) = 0$ in $H_n (|\Ncal (\Vcal)|; G)$. Here $G=\Z$ if $M$ is orientable, and $G = \Z_2$ if $M$ is non-orientable. 
 By Lemma~\ref{Guth}, $\FillRad(M, g) \leq \frac{1}{8} \inj(M, g)$. On the other hand, Theorem~\ref{thm:FillRad_Sys} yields that 
  \begin{equation*}
   \inj(M, g) \leq \frac{1}{2} \sys \pi_1 (M, g) \leq  3 \FillRad (M, g) \leq \frac{3}{8} \inj(M, g)  
 \end{equation*}
holds for every closed essential Riemannian manifold $(M, g)$, thus leading to a contradiction. 
\end{proof}

\section{Gromov's covering argument}
\label{sec:gromov_covering}

In this section, we introduce Gromov's covering technique, which first appeared in \cite[Section 5.3.]{Gromov1983}.
This section is an adaptation of Gromov's work. 
Moreover, some additional explanations are provided to clarify the original arguments. We also refer to \cite{Chen2019} for the application to Betti numbers.

\medskip

Let $\theta > 0$ and $R_0 > 0$.

\begin{definition}
A metric ball $B(p, R)$ in an $n$-dimensional Riemannian manifold $(M, g)$ is called \emph{$\theta$-admissible}
if either $R = R_0$ and $\vol_{g} \left( B(p, 5R_0) \right) \leq 5^{n + \theta} \vol_{g} ( B(p, R_0 ) )$, or $0 < R < R_0$ and the following two conditions hold:
 \begin{enumerate}
    \item[(1)] $ \vol_{g} ( B(p, 5R ) ) \leq 5^{n + \theta} \vol_{g} (B(p, R)), $
    \item[(2)] $ \vol_{g} ( B(p, 5R^{\prime}) ) > 5^{n + \theta} \vol_{g} ( B(p, R^{\prime}) ) $ whenever $R < R^{\prime} \leq R_0 $.
  \end{enumerate} 
\end{definition}   

In Proposition~\ref{prop:admissible_existence} of Appendix A, we show that at any point in a complete Riemannian manifold $M$, there exists a unique $\theta$-admissible ball.

For appropriately chosen positive constants $\theta$ and $R_0$, we use $\theta$-admissible metric balls to construct an open cover of the Riemannian manifold $(M, g)$. Throughout this section, the constant $R_0$ is chosen to be 
  \begin{equation} \label{equ:R_0}
   R_0 = \frac{1}{8n + 16} \, \inj(M, g).
  \end{equation} 
The constant $\theta$ is then defined as
 \begin{equation} \label{equ:theta}
  \theta = \sqrt{ \log_5{ \frac{\vol_{g}(M)}{\beta_n \ts R_0^{n}} } },
  \end{equation}
 where $\beta_n$ is the constant from Croke's local embolic inequality (\ref{Croke}). Then $\theta$ must be positive, since $\vol_g(M) > \vol_g(B(p, R_0)) \geq \beta_n R_0^n$ by (\ref{Croke}), and $B(p, R_0) \neq M$ according to our choice of $R_0$.

We next show that a cover of the manifold $M$ can be constructed using the $\theta$-admissible balls.  

\begin{lemma} \label{lem:covering_trick}
 Let $M$ be a closed connected Riemannian manifold. Then there exists a finite collection $\{ B(p_j, R_j) \}_{j=1}^N$ of pairwise disjoint $\theta$-admissible balls, indexed so that $R_1 \geq R_2 \geq \cdots \geq R_N$, such that 
  \[ \Ucal = \{ B(p_j, 2R_j) \}_{j=1}^N \]
 is an open cover of $M$.  
\end{lemma}

\begin{proof}
We define the radius function $R: M \to (0, + \infty)$ by assigning $R(p)$ the radius of the unique $\theta$-admissible ball centered at $p$.
By Lemma~\ref{lem:admissible_prop}, the function $R$ is upper semicontinuous. Hence the function $R$ attains its maximum on every compact subset.

Now we set $K_1 = M$. Then choose $p_1 \in K_1$ such that $R_1 = R(p_1)$ is the maximum of $R$ on $K_1$. 
Suppose that pairwise disjoint balls
 \[ B(p_1, R_1), \cdots B(p_{j-1}, R_{j-1}) \]
have already been chosen. 
Set $K_{j} = M \setminus \bigcup_{i = 1}^{j-1} B(p_i, 2R_i)$. If $K_j \neq \emptyset$, then choose $p_j \in K_j$ 
maximizing the radius function $R(p)$ on $K_j$. 
Set $R_j = R(p_j)$.
We add the ball $B(p_j, R_j)$ to the collection. 
Continue this process until no further balls can be added. 
By Proposition~\ref{prop:admissible_existence},
at any point in $M$ there exists a unique $\theta$-admissible ball. Moreover, Lemma~\ref{lem:admissible_prop} implies that the radius of $\theta$-admissible balls has a positive lower bound when $M$ is compact. 
Let $\rho = \inf \{ R(p) :  p \in M\} > 0$. Then
 \[ \dist_g(p_i, p_j) \geq 2R_i \geq 2\rho \]
for every $i < j$.  
Therefore all the chosen centers $p_j$ are $2\rho$-separated. 
Hence the above procedure terminates after finitely many steps. 
By the construction, $R_1 \geq R_2 \geq \cdots \geq R_N$. 
The above construction immediately yields that the collection $\{ B(p_j, 2R_j) : j = 1, \cdots , N \}$ is an open cover of $M$. By the construction, the procedure continues whenever $K_j = M \setminus \bigcup_{i<j} B(p_i, 2R_i)$ is nonempty. Therefore, when the process terminates, we must have $K_{N+1} = \emptyset$. Hence,
 \[ M = \bigcup_{j=1}^N B(p_j, 2R_j). \]

Next we show that balls in the collection
 \[ \{ B(p_1, R_1), \cdots , B(p_N, R_N) \} \]
are mutually disjoint. Note that at the $j$th step, we choose
 \[ p_j \in K_j = M \setminus \bigcup_{i < j} B(p_i, 2R_i). \]
Thus, for every $i < j$,
 \[ \mathrm{dist}_g (p_i, p_j) \geq 2R_i. \]
Since the compact sets $K_j$ are nested, their maximal radii satisfy $R_j \leq R_i$.  
Consequently, 
 \[ \mathrm{dist}_g (p_i, p_j) \geq 2R_i \geq R_i + R_j. \]
Therefore, the open balls $B(p_i, R_i)$ and $B(p_j, R_j)$ are disjoint.      
\end{proof}

\begin{remark}
 In Berger's covering trick lemma (see \cite[Lemma 125, \S 7.2]{Berger2003}), a much simpler technique is used to construct an open cover using metric balls with the same radius. In Lemma~\ref{lem:covering_trick}, the construction of an open cover by admissible balls is more delicate.    
\end{remark}

Given a $\theta$-admissible metric ball $B(p, R)$, for convenience,
choose the unique nonnegative integer $k$ such that
\begin{equation*}
 5^{-k} R_0 \leq R < 5^{-k+1} R_0 .
\end{equation*}
Recall that the constant $\theta$ is chosen in (\ref{equ:theta}). Then the exponent $k$ is controlled by $\theta$.  

\begin{lemma} \label{lem:admissible}
 The exponent $k$ satisfies $k < \theta$. 
\end{lemma}

\begin{proof}
If $k=0$, then the result is immediate because $\theta > 0$. Now assume $k \geq 1$.
By Croke's local embolic inequality (\ref{Croke}), we have 
\begin{align*}
 \vol_{g} (M) & \geq \vol_{g} (B(p, 5R_0))  \\ 
  & > 5^{n+\theta} \vol_{g} ( B(p, R_0) ) \\ 
  & \hspace{20pt} \vdots \\ 
  & > 5^{k(n+\theta)} \vol_{g} ( B(p, 5^{-k + 1} R_0 ) ) \\
  & > 5^{k(n+\theta)} \vol_{g} ( B(p, R) ) \\
  & \geq 5^{k(n+\theta)} \ts \beta_n R^n \\
  & \geq 5^{k(n+\theta)} \beta_n ( 5^{-k} R_0 )^n .
\end{align*}
Hence we obtain an upper bound on $k$,
\begin{align*}
 k < & \, \frac{1}{\theta} \cdot \log_5{\frac{ \vol_{g} (M) }{\beta_n R_0^n} }   \\
    = & \, \theta.
\end{align*}
\end{proof}

The triangle inequality implies the following useful fact. 

\begin{lemma} \label{lem:5_radius}
 In the open cover $\Ucal = \{ B(p_j, 2R_j) : j = 1, 2, \cdots , N \}$, if $B(p_i, 2R_i) \cap B(p_j, 2R_j) \neq \emptyset$ and $i \geq j$, then $B(p_i, R_i) \subseteq B(p_j, 5R_j)$.
\end{lemma}

Denote by $T$ the number of pairwise intersections in the open cover $\Ucal$. Hence $T$ is the number of edges in the 1-skeleton of the nerve $\mathcal{N}(\mathcal{U})$.    
 For $1 \leq j \leq N$, suppose that in the cover $\Ucal$, each ball in the subcollection $ \{ B(p_{j_1}, 2R_{j_1})$, $B(p_{j_2}, 2R_{j_2})$, $\cdots $ , $B(p_{j_{T_j}}, 2R_{j_{T_j}}) \} $ has nonempty intersection with the ball $B(p_j, 2 R_j)$, and $j_i > j$ for $1 \leq i \leq T_j$. Then we have $ T = T_1 + T_2 + \cdots + T_{N-1} $. 
 
\begin{proposition} \label{prop:gromov_covering}
 Let $M$ be a closed smooth $n$-dimensional manifold. 
 Then there exists a Riemannian metric $g$ on $M$, such that the number $T$ of pairwise intersections in the open cover $\Ucal$ is bounded above in terms of the embolic volume $\emb(M)$ as follows,
  \[ T \leq C_n \emb(M) \exp{\left( C_n^{\prime} \sqrt{\log{(1+\emb(M)})} \right)},\]
 where $C_n$ and $C_n^{\prime}$ are positive constants depending only on $n$.   
\end{proposition} 

The proof of Proposition~\ref{prop:gromov_covering} is an adaptation of Gromov's argument in \cite[Section 5.3.]{Gromov1983}. For completeness, we include the proof.

\begin{proof}
In the construction of $\theta$-admissible metric balls $B(p_j, R_j) \subseteq M$, we assume that there exists a nonnegative integer $k_j$, such that
 \[ 5^{-k_j}R_0 \leq R_j < 5^{-k_j + 1} R_0 , \]
where the exponent $k_j$ satisfies $k_j < \theta$, see Lemma~\ref{lem:admissible}.

By Lemma~\ref{lem:5_radius} and Croke's local embolic inequality (\ref{Croke}), we have the following estimate for the number $T$ of pairwise intersections of balls in $\Ucal$, 
\begin{align*}
  \vol_{g}(M) & \geq \sum_{j = 1}^N \vol_{g}( B(p_j, R_j) ) \\ 
  & \geq 5^{-n-\theta} \sum_{j=1}^N \vol_{g} ( B(p_j, 5R_j) ) \\  
  & \geq 5^{-n-\theta} \sum_{j=1}^N \sum_{i = 1}^{T_j} \vol_{g} ( B( p_{j_i}, R_{j_i} ) )  \\   
  & \geq 5^{-n-\theta} \sum_{j = 1}^N \sum_{i = 1}^{T_j} \beta_n R_{j_i}^n \\        
  & \geq 5^{-n-\theta} \sum_{j = 1}^N \sum_{i = 1}^{T_j} \beta_n \left( 5^{-k_{j_i}} R_0 \right)^n \\ 
  & > 5^{-n-\theta} \ts T \ts \beta_n 5^{ - n \cdot \sqrt{ \log_5{ \frac{\vol_{g}(M)}{\beta_n R_0^n} } } } R_0^n .
\end{align*}
Hence we have
\begin{align*}
 T & \leq \frac{\vol_{g}(M)}{\beta_n R_0^n} 5^{  n \cdot \sqrt{ \log_5{ \frac{\vol_{g}(M)}{\beta_n R_0^n}  } } } \cdot 5^{n+\theta} \\
     & \leq \frac{\vol_{g}(M)}{\beta_n R_0^n} 5^{n +  (n + 1) \cdot \sqrt{ \log_5 { \frac{\vol_{g}(M)}{\beta_n R_0^n} } }} . 
\end{align*}
Therefore,
 \[ T \leq  C_n \frac{\vol_g(M)}{\inj(M, g)^n} \exp{\left( C_n^{\prime} \sqrt{ \log{\left( 1 + \frac{\vol_g(M)}{\inj(M, g)^n} \right) } } \, \right)}, \]
where $C_n$ and $C_n^{\prime}$ are positive constants depending only on $n$.

We claim that there exists a Riemannian metric $g_0$ on $M$ for which the associated cover satisfies  
 \begin{equation} \label{equ:T_Emb}  
  T \leq  C_n \emb(M) \exp{\left( C_n^{\prime} \sqrt{ \log{(1+ \emb(M))} } \, \right)}. 
  \end{equation} 
Set
 \[ \emb(M, g) = \frac{\vol_g(M)}{\inj(M, g)^n} , \]
thus
 \[ \emb(M) = \inf_g \emb(M, g), \]
where the infimum runs over all Riemannian metrics $g$ on $M$.    
Then the preceding estimate becomes  
 \[ T \leq C_n \emb(M, g) \exp{\left( C_n^{\prime} \sqrt{\log{(1+ \emb(M, g))}} \right)}. \]
Now we choose a Riemannian metric $g_0$ such that
 \[ \emb(M, g_0) \leq 2 \emb(M). \]   
The metric $g_0$ exists by the definition of the infimum. After adjusting the dimension-dependent constants, this immediately gives the desired estimate.
\end{proof}

Let $t_k$ denote the number of $k$-simplices of the nerve $\Ncal(\Ucal)$, $0 \leq k \leq n$. 
According to the definition of nerve, $t_1$ is equal to the number $T$ of pairwise intersections of balls in $\Ucal$. 
\begin{lemma} \label{lem:count_n_simplices} 
 Let $\Ucal$ be the open cover given by
 Proposition~\ref{prop:gromov_covering}.    
 Then the total number of simplices of $\Ncal(\Ucal)^{(n)}$ satisfies
  \[ \sum_{k=0}^n t_k \leq C_n \, (1+\emb(M))^{n+1} \, \exp{\left( C_n^{\prime} \sqrt{\log(1 + \emb(M))} \right)}. \]
\end{lemma}
\begin{proof}
Since $M$ is connected and $\Ucal$ is a cover of nonempty sets, the $1$-skeleton of $\Ncal(\Ucal)$ is connected. If $N$ denotes the number of vertices, then 
 \[ N \leq t_1 + 1 .\]
Consequently,
 \begin{align*}
  \sum_{k=0}^n t_k \leq & \sum_{k=0}^n \left( 
   \begin{array}{c}
    N \\
    k+1
   \end{array}
  \right) \\
  \leq & \sum_{k=0}^n \left( 
   \begin{array}{c}
    t_1 + 1 \\
    k+1
   \end{array}
  \right)  \\
  \leq & (n+1) (t_1 + 1)^{n+1}. 
 \end{align*} 
By Proposition~\ref{prop:gromov_covering}, there exists a Riemannian metric $g$ on $M$ such that
 \[ t_1 \leq C_n \emb(M) \, \exp{\left( C_n^{\prime} \sqrt{\log(1 + \emb(M))} \right)}. \]
After changing the constants $C_n$ and $C_n^{\prime}$, we obtain
 \[ \sum_{k=0}^n t_k \leq C_n \, (1+\emb(M))^{n+1} \, \exp{\left( C_n^{\prime} \sqrt{\log(1 + \emb(M))} \right)}. \]    
\end{proof}

% \bigskip

\section{Proof of the main theorem} 

In this section, we prove the main result (Theorem~\ref{thm_main}) of this paper.

Recall from Section~\ref{sec:gromov_covering} that, in Gromov's covering argument, we construct an open cover $\Ucal$ of $M$ using $\theta$-admissible metric balls. 
More precisely, we use the finite collection constructed in Lemma~\ref{lem:covering_trick}.

\begin{proposition} \label{prop:domination}
 Let $M$ be a closed $n$-dimensional Riemannian manifold, 
 and let $ \Ucal = \{ B(p_j, 2R_j) : j = 1, 2, \cdots , N \}$ be the open cover constructed in Section 3. Suppose that 
  \[ R_j \leq R_0 =  \frac{1}{8n + 16} \inj(M, g)  \]
  for every $j$.
  Then there exist continuous maps $f: M \to |\Ncal(\Ucal)^{(n)}|$ and $\psi: |\Ncal(\Ucal)^{(n)}| \to M$ such that $\psi \circ f$ is homotopic to the identity map. Moreover, if $\iota: |\Ncal(\Ucal)^{(n)}| \to |\Ncal(\Ucal)|$ denotes the inclusion, then $\iota \circ f$ is homotopic to the Alexandrov map $\Phi$. 
\end{proposition}

Theorem \ref{thm_main} follows from Proposition~\ref{prop:domination} and Gromov's covering argument in Section 3.  
 
\begin{proof}[Proof of Theorem~\ref{thm_main}]
Choose the metric $g$ and the cover $\Ucal$ as in Proposition~\ref{prop:gromov_covering}. Then  
 \[ t_1 \leq C_n \emb(M) \exp{\left( C_n^{\prime} \sqrt{\log{\left( 1 + \emb(M) \right)}} \right)},  \]
where $t_1$ is the number of $1$-simplices of the nerve $\Ncal(\Ucal)$.

By Proposition~\ref{prop:domination}, there exist continuous maps $f: M \to |\Ncal(\Ucal)^{(n)}|$ and $\psi: |\Ncal(\Ucal)^{(n)}| \to M$ such that $\psi \circ f$ is homotopic to the identity. Thus $M$ is dominated by $|\Ncal(\Ucal)^{(n)}|$.  

Let $t_k$, $0 \leq k \leq n$, denote the number of $k$-simplices of $\Ncal(\Ucal)$. By Lemma~\ref{lem:count_n_simplices}, the total number $\sum_{k=0}^n t_k$ of simplices in $|\Ncal(\Ucal)^{(n)}|$ satisfies 
 \[ \sum_{k=0}^n t_k \leq C_n (1 + \emb(M) )^{n+1} \exp{\left( C_n^{\prime} \sqrt{\log{\left( 1 + \emb(M) \right)}} \right)}, \]
where $C_n$ and $C_n^{\prime}$ depend only on $n$. This proves the theorem.   
\end{proof}

The proof of Theorem~\ref{thm_02} also relies on Proposition~\ref{prop:domination} and Gromov's covering argument from Section 3.

\begin{proof}[Proof of Theorem~\ref{thm_02}]
Let $g$ be the Riemannian metric given in Proposition~\ref{prop:gromov_covering}. 
If we choose the open cover $\Ucal = \{ B(p_j, 2R_j) : j = 1, 2, \cdots , N \}$ of $\theta$-admissible metric balls in the Riemannian manifold $(M, g)$, as described above, 
then Lemma~\ref{lem:count_n_simplices} gives the required upper bound for the total number of simplices in $\Ncal(\Ucal)^{(n)}$.

Let $\iota: |\Ncal(\Ucal)^{(n)}| \to |\Ncal(\Ucal)|$ be the inclusion map.
By Proposition~\ref{prop:domination}, there exists a continuous map $f: M \to |\Ncal(\Ucal)^{(n)}|$ such that $\iota \circ f$ is homotopic to the Alexandrov map $\Phi$. Moreover, Proposition~\ref{nerve_homology} implies that $\Phi_{*}([M]) \neq 0$, since $M$ is essential, and the radii of the cover elements satisfy $2R_j \leq 2R_0 = \frac{1}{4n+8} \inj(M, g) \leq \frac{1}{8} \inj(M, g)$.  
Then
 \[ (\iota \circ f)_{*}([M]) = \iota_{*} ( f_{*} ([M]) ) = \Phi_{*}([M]) \neq 0. \]
Therefore,  $f_{*}([M]) \neq 0$. This proves Theorem~\ref{thm_02}. 
\end{proof}

We now prove Proposition~\ref{prop:domination}. 
The method is similar to that used in the proof of Lemma 7 in Guth \cite{Guth2011}. We construct a map $\psi$ from the underlying space $|\Ncal(\Ucal)|$ of the nerve to the manifold $M$. 
As noted by Guth, this construction goes back to Gromov, see \cite[pp. 85--86]{gromov_bounded_cohomology_1982}.

\begin{proof}[Proof of Proposition~\ref{prop:domination}]
By the definition of the nerve, 
each vertex $u_j$ of $\mathcal{N}(\mathcal{U})$ corresponds to a ball $B(p_j, 2R_j)$ in the open cover $\mathcal{U}$. Define a map $\psi$ on the set of vertices by $\psi(u_j) = p_j$.  
We next extend $\psi$ to be a continuous map on the n-skeleton of $|\Ncal(\Ucal)|$.

Fix a total ordering on the vertex set $\{u_i : i = 1, 2, \cdots , N \}$ of $\Ncal(\Ucal)$ such that 
 \[ u_i \leq u_j \]
if the radii of the corresponding balls in $M$ satisfy $R_i \leq R_j$. Moreover, when $R_i = R_j$, we define $u_i \leq u_j$ if $i < j$. 
Let $\sigma = [u_{i_0}, \cdots , u_{i_k}], k \leq n$ be a k-simplex of $\Ncal(\Ucal)$, where the vertices are written in the fixed order defined above. Thus $ R_{i_0} \leq \cdots \leq R_{i_k}$. Denote by $\mathrm{dist}_g( \, , \, )$ the distance function induced by the Riemannian metric $g$. 
Because the corresponding balls $B(p_{i_j}, 2R_{i_j})$ in the cover $\Ucal$ have nonempty common intersection, 
 \[ \mathrm{dist}_g(p_{i_{\ell}}, p_{i_k}) \leq 2R_{i_{\ell}} + 2R_{i_k} \leq 4 R_{i_k} \]
for every $0 \leq \ell < k$. 

Now for the $k$-simplex $\sigma = [u_{i_0}, \cdots , u_{i_k}], k \leq n$ in $\Ncal(\Ucal)$, we define a corresponding ordered straight k-simplex $s_k = [p_{i_0}, \cdots p_{i_k}]: \triangle^k \to M$ inductively.   
For $k=0$, it is the constant map with value $p_{i_0}$. Assume that $s_{k-1} = [p_{i_0}, \cdots , p_{i_{k-1}}]$ has been defined, and 
 \[ \dist_g(p_{i_{k-1}}, s_{k-1}(z)) \leq 4(k-1)R_{i_{k-1}} \]
holds for every $z \in \triangle^{k-1}$.  
Furthermore, we claim that
 \[ \dist_g(p_{i_k}, s_{k-1}(z)) < \inj(M, g). \]
Indeed, 
\begin{align*}
  \mathrm{dist}_g(s_{k-1}(z), p_{i_k}) \leq & \mathrm{dist}_g(s_{k-1}(z), p_{i_{k-1}}) + \mathrm{dist}_g(p_{i_{k-1}}, p_{i_k}) \\
                        \leq & 4(k-1) R_{i_{k-1}} + 4R_{i_k} \\
                        \leq & 4k R_{i_k} \\
                        < & \inj(M, g),
 \end{align*}
 where the last inequality is true since our assumption of $R_{i_k} \leq R_0 = \frac{1}{8n+16} \inj(M, g)$ for each $i_k$, see (\ref{equ:R_0}). 
Therefore, there is a unique minimizing geodesic segment between $p_{i_k}$ and every point $s_{k-1}(z)$ in the image of the $(k-1)$-simplex $s_{k-1}$. We join $p_{i_k}$ to $s_{k-1}(z)$ by the unique minimizing geodesic segment, and use these geodesic segments to form the cone. More explicitly, let $z \in \triangle^{k-1}, t \in [0, 1]$, and let $e_k$ be the $k$-th vertex of $\triangle^k$. Define the straight $k$-simplex $s_k$ by
 \[ s_k((1-t)z + t e_k) = \exp_{p_{i_k}}{\left( (1-t) \exp_{p_{i_k}}^{-1}(s_{k-1}(z)) \right)}. \]   
We claim that
 \[ \mathrm{dist}_g (s_k(w), p_{i_k}) \leq 4 k R_{i_k} \] 
for every $w\in \triangle^k$. 
Indeed, for $w = (1-t)z + te_k$, the definition yields 
 \[ \mathrm{dist}_g(s_k(w), p_{i_k}) = (1-t) \mathrm{dist}_g (s_{k-1}(z), p_{i_k}) \leq \mathrm{dist}_g(s_{k-1}(z), p_{i_k}) < \inj(M, g).  \] 
Therefore every point of the unique minimizing geodesic segment remains in $B(p_{i_k}, 4kR_{i_k})$, proving the claim.

Because the global ordering of vertices is used in the above construction of straight k-simplex, these ordered straight simplices agree on their common faces. Identify every simplex $\sigma \in \Ncal(\Ucal)$ with $\triangle^k$, where $k \leq n$. We then define $\psi: |\Ncal(\Ucal)^{(n)}| \to M$ by
 \[ \psi|_{\sigma} = s_k. \]
This gives a continuous map.   

We next construct a map $f: M \to |\Ncal(\Ucal)^{(n)}|$ and prove that $\psi \circ f$ is homotopic to the identity. 
After sufficiently many barycentric subdivisions of a finite triangulation $\mathcal{T}$ of $M$, the closed star of every vertex $v \in \mathcal{T}$ is contained in an element $B(p_{j(v)}, 2R_{j(v)})$ of $\Ucal$. 
Define $f$ on the vertices of $\mathcal{T}$ by $f(v) = u_{j(v)}$, and extend $f$ simplicially. Let $[v_0, \cdots , v_k]$ be a simplex in the triangulation $\mathcal{T}$.
Then the closed-star condition implies that the simplex $[v_0, \cdots , v_k]$ is contained in every ball $B(p_{j(v_{\ell})}, 2R_{j(v_{\ell})})$ of $\Ucal$, thus these $(k+1)$ balls have nonempty common intersection. Consequently, the vertices $u_{j(v_0)}, \cdots u_{j(v_k)}$ span a simplex of $\Ncal(\Ucal)$. Since the dimension of the simplicial complex $\mathcal{T}$ is $n$, the map $f$ takes values in $|\Ncal(\Ucal)^{(n)}|$. Now we show that $\iota \circ f$ is homotopic to the Alexandrov map $\Phi$, where $\iota: |\Ncal(\Ucal)^{(n)}| \to |\Ncal(\Ucal)|$ is the inclusion map. 
Indeed, for every $x \in M$, each vertex in the nerve occurring with nonzero barycentric coefficients in either $f(x)$ or $\Phi(x)$ corresponds to a ball in the cover $\Ucal$ containing $x$. These vertices therefore span a common simplex of the nerve, and the straight-line homotopy joining $f(x)$ to $\Phi(x)$ in barycentric coordinates provides the homotopy between $\iota \circ f$ and $\Phi$.

It remains to prove that $\psi \circ f$ is homotopic to the identity. Let $x$ be in a simplex $[v_0, \cdots , v_k]$ of the triangulation $\mathcal{T}$. Let $u_{*}$ be the largest vertex, in the fixed total ordering, among the vertices of the simplex $f([v_0, \cdots , v_k])$, and let $p_{*}$ and $R_{*}$ denote the corresponding center and radius. Since the simplex $[v_0, \cdots , v_k]$ is contained in the ball $B(p_{*}, 2R_{*})$, we have $\mathrm{dist}_g (x, p_{*}) < 2 R_{*}$. On the other hand, by the preceding construction, 
 \[ \psi (f(x)) \in B(p_{*}, 4kR_{*}) \subseteq B(p_{*}, 4nR_{*}). \]
It follows that
 \[ \mathrm{dist}_g(x, \psi(f(x))) \leq (4n + 2) R_{*} < \inj(M, g). \]
Thus $x$ and $\psi(f(x))$ are joined by a unique minimizing geodesic. Define 
 \[ H(x, t) = \exp_x(t \cdot \exp_x^{-1}(\psi(f(x))) ). \]
The preceding uniform distance estimate implies that $H$ is continuous and $H(x, 0) = x$, $H(x, 1) = \psi(f(x))$. Therefore, $\psi \circ f$ is homotopic to the identity map on $M$.      
\end{proof}

 \bigskip

\appendix

\section{Existence of admissible balls}
\label{sec:appendix}

Let $(M, g)$ be an n-dimensional Riemannian manifold, and
let $\mathrm{Sc}_p$ denote the scalar curvature at a point $p \in M$. 
For the volume of balls with small radii in $(M, g)$, we have the following expansion in terms of the scalar curvature $\mathrm{Sc}_p$. 
\begin{theorem}[see {\cite[Theorem 3.98, p. 168]{GallotHL2004}}] \label{ball_loc}
 Let $B(p, r)$ be a ball centered at $p$ with radius $r$ in an $n$-dimensional Riemannian manifold $(M, g)$. As the radius $r$ approaches zero, the volume of $B(p, r)$ satisfies
 \begin{equation} \label{ball}
  \vol_{g} (B(p, r)) = r^n \omega_n \left[ 1 -  \frac{\text{Sc}_p}{6 (n + 2)} r^2 + o(r^3)  \right] , 
 \end{equation} 
 where $\omega_n$ is the volume of the unit ball in Euclidean $n$-space. 
\end{theorem} 

The following lemma concerns the growth of the volume of balls in a Riemannian manifold. It is a slight modification of \cite[Lemma 1]{Guth2011}.
\begin{lemma} \label{admissible}
 Let $p \in M$ be any point, and let $R_0 > 0$ be a fixed constant. 
Then for any $\theta > 0$, there exists at least one $R \in (0, R_0]$, such that 
 \begin{equation}\label{ball_growth}
  \vol_{g} (B(p, 5 R)) \leq 5^{n+\theta} \vol_{g} ( B(p, R) ) .
 \end{equation}
\end{lemma}

\begin{proof}
 We define the density function of radius $r$ at $p \in M$ as
 \begin{equation} \label{defn:density}
  \dens_p ( r ) =  \frac{ \vol_{g} (B(p, r)) }{r^n} . 
 \end{equation}
According to Theorem \ref{ball_loc}, we have 
 \begin{equation}\label{limit_density}
  \lim_{r \to 0} \dens_p (r) = \omega_n .
 \end{equation}
If we assume that (\ref{ball_growth}) is not true, then 
 \begin{align*}
  \vol_{g} (B(p, 5R)) > 5^{n+\theta} \vol_{g} ( B(p, R) ) 
 \end{align*}
 holds for all $0< R \leq R_0$, which implies
 \begin{equation*}
  \dens_p (5R) > 5^{\theta} \dens_p (R) .
 \end{equation*}
 Iterating this inequality yields:
 \begin{align*}
  \dens_p (5R_0) & > 5^{\theta} \dens_p (R_0) \\ 
   & > 5^{2 \theta} \dens_p (5^{-1} R_0 ) \\
   & \hspace{20pt} \vdots \\
   & > 5^{ \ell \ts \theta } \dens_p ( 5^{-\ell + 1} R_0 ) .
 \end{align*}
 However, this leads to a contradiction. As $\ell \to \infty$, identity (\ref{limit_density}) implies that the right-hand side of 
 \begin{equation} \label{equ:appendix_Lem2_5R0}
  \dens_p(5 \tts R_0)  > 5^{\ell \tts \theta} \dens_p ( 5^{- \ell + 1} R_0 ) ,
 \end{equation}
 tends to infinity.
\end{proof}

We next show that, for any $\theta > 0$ and any fixed constant $R_0 > 0$, there exists a unique $\theta$-admissible ball $B(p, R)$ centered at $p$ with radius $R \leq R_0$.
  
\begin{proposition} \label{prop:admissible_existence}
 For any point $p$ in an $n$-dimensional Riemannian manifold $(M, g)$, there exists a unique $\theta$-admissible metric ball $B(p, R)$.  
\end{proposition}

\begin{proof}
 Let
  \begin{equation} \label{equ:admi_defn}
   R = \sup_{r} \{ \left. 0 < r \leq R_0 \right| \vol_{g}(B(p, 5r)) \leq 5^{n + \theta} \vol_{g} (B(p, r)) \}.
  \end{equation}
Lemma~\ref{admissible} implies that the supremum $R$ exists, and $R > 0$. Then $R \in (0, R_0]$. Next we prove $B(p, R)$ is a $\theta$-admissible ball.  
Since the function 
 \[ r \mapsto \vol_g(B(p, 5r)) - 5^{n+\theta} \vol_g (B(p, r)) \]
is continuous, a sequence $r_i$ in the set
 \[  \{ 0 < r \leq R_0 : \vol_g(B(p, 5r)) \leq 5^{n+\theta} \vol_g(B(p, r)) \} \]
with $r_i \to R$ gives
 \[ \vol_g(B(p, 5R)) \leq 5^{n+\theta} \vol_g(B(p, R)). \]
If $R < R_0$, then we have
 \[ \vol_g(B(p, 5R^{\prime})) > 5^{n+\theta} \vol_g(B(p, R^{\prime})) \] 
for any $R < R^{\prime} \leq R_0$. Therefore, according to the definition, $B(p, R)$ is an admissible ball. 

The maximality property in the definition of $R$ also proves the uniqueness. Now suppose $\tilde{R}$ is another radius of an admissible ball centered at $p$. If $\tilde{R} = R_0$, then we must have $R = R_0 = \tilde{R}$ since $R$ is the supremum of all the radii satisfying the first admissibility condition.  
We may therefore assume that $\tilde{R} < R_0$. 
Since $\tilde{R}$ satisfies the first admissibility condition 
 \[ \vol_g(B(p, 5\tilde{R})) \leq 5^{n+\theta} \vol_g (B(p, \tilde{R})), \]
we have $\tilde{R} \leq R$. If $\tilde{R} < R$, then the second admissibility condition for $\tilde{R}$, evaluated at $R$, gives
 \[ \vol_g(B(p, 5R)) > 5^{n+\theta} \vol_g(B(p, R)), \]
contradicting the fact that $R$ satisfies the first admissibility condition. Hence we must have $\tilde{R} = R$.   
\end{proof}

Define $R: M \to (0, R_0]$ by assigning $R(p)$ the unique radius of $\theta$-admissible balls centered at $p \in M$. 
\begin{lemma} \label{lem:admissible_prop}
 Let $M$ be a closed Riemannian manifold. Then
 \begin{enumerate}
  \item $R$ is upper semicontinuous.
  \item $R$ has a positive lower bound on $M$.
 \end{enumerate}
\end{lemma}

\begin{proof}
 \begin{enumerate}
  \item
To prove that the function $R$ is upper semicontinuous, we define 
 \[ F(p, r) = \vol_g(B(p, 5r)) - 5^{n+\theta} \vol_g(B(p, r)). \]
Then the function $F$ is continuous in variables $(p, r)$. Moreover, by Proposition~\ref{prop:admissible_existence},
 \[ R(p) = \sup \{ r \in (0, R_0] : F(p, r) \leq 0 \}.  \]
Suppose $p_m \to p$, and pass to a subsequence such that
 \[ R(p_m) \to L = \limsup_{m \to \infty} R(p_m). \]
If $L = 0$, then $L \leq R(p)$ is immediate. Now suppose $L > 0$.  
Since $F(p_m, R(p_m)) \leq 0$, the continuity implies $F(p, L)\leq 0$. Hence we have $ L \leq R(p)$.
Therefore, 
 \[ \limsup_{m \to \infty} R(p_m) \leq R(p), \]
and this gives the conclusion that the function $R$ is upper semicontinuous.   

  \item Note that the volume function $\mathrm{V}: M \to \R$ of balls defined by 
  \begin{equation} \label{equ:vol_ball}
    \mathrm{V}(p) = \vol_g(B(p, 5R_0)) 
  \end{equation}
 is continuous. Hence if we suppose the infimum of the radii of $\theta$-admissible balls is zero, then there exists a sequence $\{ p_i \in M | i = 1, 2, \cdots  \}$ of points in $M$, such that the radius $R_i$ of $\theta$-admissible balls at $p_i$ satisfies
 \[ \lim_{i \to \infty} R_i = 0. \]
 Choose $k(i) \in \mathbb{N}$ such that $R_i < 5^{-k(i)}R_0$. Since $R_i \to 0$, we may further assume $k(i) \to \infty$.  
Now consider the density function (\ref{defn:density}) defined in the proof of Lemma~\ref{admissible}. Since every radius $R_0, 5^{-1}R_0, \cdots , 5^{-k(i)}R_0 $ exceeds $R_i$, iterating the second admissibility condition gives 
  \begin{equation} \label{ineq:radius_admissible}
   \dens_{p_i}(5R_0) >  5^{( k(i) + 1) \, \theta} \dens_{p_i}(5^{-k(i)}R_0)
  \end{equation}
 holds. Here $k(i) \to \infty$ when $i \to \infty$. However, the function
  \[ p \mapsto \dens_{p}(5R_0) \] 
 has a maximum value on $M$, since the volume function $\mathrm{V}$ defined in (\ref{equ:vol_ball}) is continuous. On the other hand, for sufficiently large $i$, the radius $5^{-k(i)}R_0$ is at most $\inj(M, g) / 2$. Then Croke's local embolic inequality (\ref{Croke}) gives
  \[ \dens_{p_i}(5^{-k(i)}R_0) \geq \beta_n. \] 
Therefore the left-hand side of the inequality (\ref{ineq:radius_admissible}) has an upper bound, but the right-hand side of (\ref{ineq:radius_admissible}) tends to infinity, resulting in a contradiction.  
 
 \end{enumerate} 
\end{proof}

\bibliographystyle{amsalpha}
\bibliography{references}

@book {Berger2003,
    AUTHOR = {Berger, Marcel},
     TITLE = {A panoramic view of {R}iemannian geometry},
 PUBLISHER = {Springer-Verlag, Berlin},
      YEAR = {2003},
     PAGES = {xxiv+824},
      ISBN = {3-540-65317-1},
   MRCLASS = {53-02 (53C20 53C21 53C55 58J20)},
  MRNUMBER = {2002701},
MRREVIEWER = {J\"{u}rgen Eichhorn},
       DOI = {10.1007/978-3-642-18245-7},
       URL = {https://doi.org/10.1007/978-3-642-18245-7},
}

@misc{Chen2019,
      title={Covering trick and embolic volume}, 
      author={Chen, Lizhi},
      year={2019},
      note = {arXiv:1911.00691},
      eprint={},
      archivePrefix={},
      primaryClass={}, 
}

@article {Croke1980,
    AUTHOR = {Croke, Christopher B.},
     TITLE = {Some isoperimetric inequalities and eigenvalue estimates},
   JOURNAL = {Ann. Sci. \'{E}cole Norm. Sup. (4)},
  FJOURNAL = {Annales Scientifiques de l'\'{E}cole Normale Sup\'{e}rieure. Quatri\`eme
              S\'{e}rie},
    VOLUME = {13},
      YEAR = {1980},
    NUMBER = {4},
     PAGES = {419--435},
      ISSN = {0012-9593},
   MRCLASS = {58G25 (35P15 49G05 53C65)},
  MRNUMBER = {608287},
MRREVIEWER = {A. G. Ramm},
       URL = {http://www.numdam.org/item?id=ASENS_1980_4_13_4_419_0},
}

@incollection {crokeK2002_universal,
    AUTHOR = {Croke, Christopher B. and Katz, Mikhail},
     TITLE = {Universal volume bounds in {R}iemannian manifolds},
 BOOKTITLE = {Surveys in differential geometry, {V}ol.\ {VIII} ({B}oston,
              {MA}, 2002)},
    SERIES = {Surv. Differ. Geom.},
    VOLUME = {8},
     PAGES = {109--137},
 PUBLISHER = {Int. Press, Somerville, MA},
      YEAR = {2003},
      ISBN = {1-57146-114-0},
   MRCLASS = {53C23 (53C20 53C21 53C22)},
  MRNUMBER = {2039987},
MRREVIEWER = {Andrea\ Sambusetti},
       DOI = {10.4310/SDG.2003.v8.n1.a4},
       URL = {https://doi.org/10.4310/SDG.2003.v8.n1.a4},
}

@book {Dugundji1966,
    AUTHOR = {Dugundji, James},
     TITLE = {Topology},
 PUBLISHER = {Allyn and Bacon, Inc., Boston, MA},
      YEAR = {1966},
     PAGES = {xvi+447},
   MRCLASS = {54.00},
  MRNUMBER = {193606},
MRREVIEWER = {M.\ Edelstein},
}

@book {GallotHL2004,
    AUTHOR = {Gallot, Sylvestre and Hulin, Dominique and Lafontaine,
              Jacques},
     TITLE = {Riemannian geometry},
    SERIES = {Universitext},
   EDITION = {Third},
 PUBLISHER = {Springer-Verlag, Berlin},
      YEAR = {2004},
     PAGES = {xvi+322},
      ISBN = {3-540-20493-8},
   MRCLASS = {53-01 (53C20 53C21 53C23)},
  MRNUMBER = {2088027},
MRREVIEWER = {Joseph\ E.\ Borzellino},
       DOI = {10.1007/978-3-642-18855-8},
       URL = {https://doi.org/10.1007/978-3-642-18855-8},
}

@article {Gromov1981,
    AUTHOR = {Gromov, Michael},
     TITLE = {Curvature, diameter and {B}etti numbers},
   JOURNAL = {Comment. Math. Helv.},
  FJOURNAL = {Commentarii Mathematici Helvetici},
    VOLUME = {56},
      YEAR = {1981},
    NUMBER = {2},
     PAGES = {179--195},
      ISSN = {0010-2571,1420-8946},
   MRCLASS = {53C20 (57R15)},
  MRNUMBER = {630949},
MRREVIEWER = {Karsten\ Grove},
       DOI = {10.1007/BF02566208},
       URL = {https://doi.org/10.1007/BF02566208},
}

@article{gromov_bounded_cohomology_1982,
  title={Volume and bounded cohomology},
  author={Gromov, Michael},
  journal={Publications Math{\'e}matiques de l'IH{\'E}S},
  volume={56},
  pages={5--99},
  year={1982}
}

@article {Gromov1983,
    AUTHOR = {Gromov, Michael},
     TITLE = {Filling {R}iemannian manifolds},
   JOURNAL = {J. Differential Geom.},
  FJOURNAL = {Journal of Differential Geometry},
    VOLUME = {18},
      YEAR = {1983},
    NUMBER = {1},
     PAGES = {1--147},
      ISSN = {0022-040X},
   MRCLASS = {53C20 (53C21 57R99)},
  MRNUMBER = {697984},
MRREVIEWER = {Yu. Burago},
       URL = {http://projecteuclid.org/euclid.jdg/1214509283},
}

@article {Guth2011,
    AUTHOR = {Guth, Larry},
     TITLE = {Volumes of balls in large {R}iemannian manifolds},
   JOURNAL = {Ann. of Math. (2)},
  FJOURNAL = {Annals of Mathematics. Second Series},
    VOLUME = {173},
      YEAR = {2011},
    NUMBER = {1},
     PAGES = {51--76},
      ISSN = {0003-486X,1939-8980},
   MRCLASS = {53C23},
  MRNUMBER = {2753599},
MRREVIEWER = {St\'ephane\ Sabourau},
       DOI = {10.4007/annals.2011.173.1.2},
       URL = {https://doi.org/10.4007/annals.2011.173.1.2},
}

@misc{Harpe2017,
author        = {Pierre de la Harpe},
title         = {Brouwer Degree, Domination of Manifolds and Groups Presentable by Products},
year          = {2016},
howpublished = {\url{http://www.map.mpim-bonn.mpg.de/images/c/cf/Brouwer_degree.pdf}},
note          = {Bulletin of the Manifold Atlas. Accessed on June 21, 2026},
eprint        = {1609.06637},
archivePrefix = {arXiv},
url           = {http://www.map.mpim-bonn.mpg.de/images/c/cf/Brouwer_degree.pdf},
}

@book {Hatcher2002,
    AUTHOR = {Hatcher, Allen},
     TITLE = {Algebraic topology},
 PUBLISHER = {Cambridge University Press, Cambridge},
      YEAR = {2002},
     PAGES = {xii+544},
      ISBN = {0-521-79160-X; 0-521-79540-0},
   MRCLASS = {55-01 (55-00)},
  MRNUMBER = {1867354},
MRREVIEWER = {Donald\ W.\ Kahn},
}

@article {Weiss1996,
    AUTHOR = {Weiss, Michael},
     TITLE = {Curvature and finite domination},
   JOURNAL = {Proc. Amer. Math. Soc.},
  FJOURNAL = {Proceedings of the American Mathematical Society},
    VOLUME = {124},
      YEAR = {1996},
    NUMBER = {2},
     PAGES = {615--622},
      ISSN = {0002-9939,1088-6826},
   MRCLASS = {53C23 (53C21 55P99 57R19)},
  MRNUMBER = {1291795},
MRREVIEWER = {Joseph\ E.\ Borzellino},
       DOI = {10.1090/S0002-9939-96-03056-0},
       URL = {https://doi.org/10.1090/S0002-9939-96-03056-0},
}

@article {Whitehead1948,
    AUTHOR = {Whitehead, J. H. C.},
     TITLE = {On the homotopy type of {ANR}'s},
   JOURNAL = {Bull. Amer. Math. Soc.},
  FJOURNAL = {Bulletin of the American Mathematical Society},
    VOLUME = {54},
      YEAR = {1948},
     PAGES = {1133--1145},
      ISSN = {0002-9904},
   MRCLASS = {56.0X},
  MRNUMBER = {29504},
MRREVIEWER = {R.\ H.\ Fox},
       DOI = {10.1090/S0002-9904-1948-09140-6},
       URL = {https://doi.org/10.1090/S0002-9904-1948-09140-6},
}

%\begin{thebibliography}{99}
% \bibitem{Chen2019} L. Chen, Covering trick and embolic volume, arXiv:1911.00691 (2019).  
% \bibitem{Berger2003} M. Berger, A Panoramic View of Riemannian Geometry, Springer, Berlin, Heidelberg, 2003. 
% \bibitem{Croke1980} C. Croke, Some isoperimetric inequalities and eigenvalue estimates, Ann. Sci. \`Ecole Norm. Sup. (4) \textbf{13} (1980), no. 4, 419--435.
% \bibitem{Dugundji1966} J. Dugundji, Topology, Allyn and Bacon, 1966. 
% \bibitem{Gromov1981} M. Gromov, Curvature, diameter and Betti numbers, Comm. Math. Helv. \textbf{56} (1981), 179--195. 
% \bibitem{Gromov1983} M. Gromov, Filling Riemannian manifolds, J. Differential Geom. \textbf{18} (1983), no. 1, 1--147.   
%\bibitem{Guth2011} L. Guth, Volume of balls in large Riemannian manifolds, Ann. Math. (2) \textbf{173} (2011), 51--76. 
%\bibitem{Harpe2017} P. de la Harpe, Brouwer degree, domination of manifolds and groups presentable by products, Bull. Manifold Atlas (2017), \url{http://www.map.mpim-bonn.mpg.de/Main_Page} .
%\bibitem{Hatcher2002} A. Hatcher, Algebraic Topology, Cambridge University Press, Cambridge, UK, 2002. 
% \bibitem{Weiss1996} M. Weiss, Curvature and finite domination, Proc. Amer. Math. Soc. \textbf{124} (1996), no. 2, 615--622.
% \bibitem{Whitehead1948} J. Whitehead, On the homotopy type of ANR, Bull. AMS \textbf{54} (1948), 1133--1145.
%\end{thebibliography}
 
\end{document}